\documentclass[12pt]{amsart}
\usepackage{amsmath,amsfonts,amsthm,amscd,amssymb,stmaryrd}
\usepackage{graphicx}
\newtheorem{theorem}{Theorem}

\newtheorem{proposition}[theorem]{Proposition}

\newtheorem{corollary}[theorem]{Corollary}
\newtheorem{remark}[theorem]{Remark}
\numberwithin{equation}{section}
\numberwithin{theorem}{section}
\numberwithin{figure}{section}
\numberwithin{table}{section}
\newcommand{\ZZ}{\mathbb{Z}}

\newcommand{\RR}{\mathbb{R}}
\newcommand{\RP}{\mathbb{RP}}
\newcommand{\CP}{\mathbb{CP}}

\begin{document}
\bibliographystyle{amsalpha}
\title{The mass of the product of spheres}
\author{Jeff A. Viaclovsky}
\address{Department of Mathematics, University of Wisconsin, Madison, WI 53706}
\email{jeffv@math.wisc.edu}
\thanks{Research partially supported by NSF Grant DMS-1105187}
\date{December 18, 2013}
\begin{abstract}
Any compact manifold with positive scalar curvature has
an associated asymptotically flat metric constructed 
using the Green's function of the conformal Laplacian, 
and the mass of this metric is an important 
geometric invariant.
An explicit expression for the mass 
of the product of spheres $S^2 \times S^2$, both with the 
same Gaussian curvature, is given. 
Expressions for the masses of the quotient spaces 
$G(2,4)$, and $\RP^2 \times \RP^2$ are also given. 
The values of these masses arise
in a construction of critical metrics on certain $4$-manifolds;
applications to this problem will also be discussed.
\end{abstract}
\maketitle
\section{Introduction}
The conformal Laplacian in dimension $4$ is the operator
\begin{align}
\square u = - 6 \Delta u + Ru,
\end{align}
where the convention is to use the analyst's Laplacian (which has negative
eigenvalues). If $(M,g)$ is compact and has positive scalar
curvature, then for any $x \in M$, there exists a unique positive solution to
the equation
\begin{align}
\square G &= 0 \ \ \mathrm{ on } \ M \setminus \{x\}\\
G &= r^{-2}(1 + o(1))
\end{align}
as $r \rightarrow 0$, where $r$ is geodesic distance to the basepoint $x$,
which is called the Green's function \cite{LeeParker}.

The metric $\hat{g}_{M,x} = G^2 g_M$, defined on $M \setminus \{x\}$,
is scalar-flat and asymptotically flat (AF) of order $2$.
The mass of an AF space is defined by
\begin{align}
\label{massdef}
{\mathrm{mass}}(\hat{g}_{M,x}  ) = \lim_{R \rightarrow \infty} \omega_3^{-1}\int_{S(R)} \sum_{i,j}
( \partial_i g_{ij} - \partial_{j} g_{ii} ) ( \partial_j \ \lrcorner \ dV),
\end{align}
with $\omega_3 = Vol(S^3)$.
This quantity does not depend upon the coordinate system chosen at infinity,
thus is a geometric invariant \cite{Bartnik}.

Let $(S^2 \times S^2, g_{S^2 \times S^2})$
be the product of $2$-dimensional spheres with unit Gauss curvature.
The product metric on $S^2 \times S^2$ admits the 
quotient $S^2 \times S^2/\ZZ^2$
where $\ZZ_2$ acts by the antipodal map on both factors.
It is  well-known that $S^2 \times S^2/\ZZ^2$ is
diffeomorphic to $G(2,4)$, the Grassmannian of $2$-planes
in $\RR^4$, see for example \cite{SingerThorpe}. Another 
quotient is $\RP^2 \times \RP^2$.
The product metric descends to an Einstein metric
on both of these quotients. Since these spaces are symmetric, 
the mass does not depend upon the base point. We will 
therefore denote
\begin{align}
m_1 &= \mathrm{mass}(\hat{g}_{S^2 \times S^2}),\\
m_2 &=   \mathrm{mass}(\hat{g}_{G(2,4)}),\\
m_3 &=  \mathrm{mass}(\hat{g}_{\RP^2 \times \RP^2}).
\end{align}

By the positive mass theorem of Schoen-Yau, $m_i > 0$ for $i = 1, 2, 3$  
\cite{SYI, SYII}.
Note that since $S^2 \times S^2$ is spin, this also follows from Witten's
proof of the positive mass theorem~\cite{Witten}.

Our main result is to give an explicit formula for these masses,
and an approximate numerical value:
\begin{theorem}
\label{t1}
The values of $m_1, m_2$, and $m_3$ may each be written as an 
explicit infinite sum (see Theorems \ref{masss2s2}, \ref{massg24}, 
and \ref{massrp2rp2} for the formulas), 
and the approximate numerical values 
of these masses are given in Table \ref{massvalues}.
\end{theorem}

\begin{table}[ht]
\caption{Mass values}
\centering
\begin{tabular}{lc}
\hline\hline
Manifold  & Approximate mass\\
\hline
$S^2 \times S^2$  & $m_1 \sim .5872$\\
$G(2,4)$  & $m_2 \sim 2.6289$\\
$\RP^2 \times \RP^2$ & $m_3 \sim 8.4323$\\
\hline
\end{tabular}\label{massvalues}
\end{table}
\subsection{Application to critical metrics}
Consider the functional
\begin{align*} 
\mathcal{B}_{t}[g] = \int |W|^2\ dV + t \int R^2\ dV,
\end{align*}
where $W$ denotes the Weyl tensor, and $R$ denotes the 
scalar curvature. 
A critical metric for $\mathcal{B}_{t}$ will also be called
a {\em{$B^t$-flat metric}}.
It is shown in \cite{GVCritical} that for $t \neq 0$, a $B^t$-flat metric on a 
compact manifold necessarily has constant scalar curvature, 
and the Euler-Lagrange equation is equivalent to 
\begin{align*}
B = 2 t R \cdot E,
\end{align*}
where $B$ denotes the {\em{Bach tensor}}, and $E$ denotes the traceless Ricci tensor. 
That is, the Bach tensor is a constant multiple of the traceless Ricci tensor.

 In \cite{GVCritical}, many examples of $B^t$-flat metrics were 
constructed using a gluing procedure roughly by gluing 
on a Green's function metric to a compact Einstein metric. 
The metrics obtained are $B^t$-flat for certain values of $t$ depending
upon the geometry of the gluing factors, and the mass of the 
Green's function metric enters into the formula for $t_0$. 
It would be too long to list all of the examples here, 
but a first application is to give approximate numerical 
values for $t_0$, see Table \ref{values} for a few examples:

\begin{theorem}\cite{GVCritical}
\label{t2}
A $B^t$-flat metric exists on the manifolds in Table~\ref{values} for 
some~$t$ near the indicated value(s) of $t_0$. 
\end{theorem}

\begin{table}[ht]
\caption{Approximate values of $t_0$}
\centering
\begin{tabular}{lc}
\hline\hline
Topology of connected sum & Value(s) of $t_0$\\
\hline
$S^2 \times S^2 \# \overline{\CP}^2$ & $-1/3, - (9 m_1)^{-1} \sim -.1892$\\
$ 2 \# S^2 \times S^2$ & $- 2(9 m_1)^{-1} \sim -.3784$ \\
$G(2,4) \#  \overline{\CP}^2$ &  $-1/3, - (9 m_2)^{-1} \sim -.0422$\\
$G(2,4) \# G(2,4)$  & $- 2(9 m_2)^{-1} \sim -.0845$\\
$\RP^2 \times \RP^2 \# \overline{\CP}^2$ & $-1/3, - (9 m_3)^{-1} \sim -.0131$\\
$\RP^2 \times \RP^2 \# \RP^2 \times \RP^2$  & $- 2(9 m_3)^{-1} \sim -.0263$\\
\hline
\end{tabular}\label{values}
\end{table}

On several connected sums, $B^t$-flat 
metrics were constructed for
possibly $2$ different values of $t_0$. Another corollary 
is that on all of the examples for which two 
values of $t_0$ were listed, these values are distinct:
\begin{theorem}\label{t3}
$B^t$-flat metrics exists on the manifolds in Table \ref{tab3} for some
$t$ near the indicated values of $t_0$. Consequently, these manifolds
admit $B^t$-flat metrics for at least $2$ different values of $t$. 
\end{theorem}

\begin{table}[ht]
\caption{Examples of manifolds admitting $B^t$-flat metrics for 
distinct values of $t$.}
\centering
\begin{tabular}{ll}
\hline\hline
Topology of connected sum& Value(s) of $t_0$ \\
\hline
$ S^2 \times S^2 \# \overline{\CP}^2 $ & $-1/3 \neq - (9 m_1)^{-1}$\\
$G(2,4) \#  \overline{\CP}^2$ & $-1/3 \neq - (9 m_2)^{-1}$ \\
$G(2,4) \#  S^2 \times S^2$ & $- 2(9 m_1)^{-1} \neq - 2(9 m_2)^{-1}$\\
$G(2,4) \# \RP^2 \times \RP^2$ &  $- 2(9 m_3)^{-1} \neq - 2(9 m_2)^{-1}$
\\
$\RP^2 \times \RP^2 \# \overline{\CP}^2$ & $-1/3 \neq -(9 m_3)^{-1}$
\\
$\RP^2 \times \RP^2 \# S^2 \times S^2$ & $- 2(9 m_1)^{-1} \neq -2(9 m_3)^{-1}$
\\
\hline
\end{tabular}\label{tab3}
\end{table}
It is remarked that the first entry in Table \ref{tab3} is 
known to admit an Einstein metric \cite{CLW}, which is $B^t$-flat
for all $t$. However, the metrics obtained in Theorem \ref{t3} are
not Einstein.

In conclusion, it is noted that the approximate numerical values listed here in the 
introduction do depend upon computer calculations, but only 
to calculate partial sums in the explicit expression for the mass, 
it is emphasized that no numerical integration is necessary. 
Rigorous error estimates are provided in Section \ref{Error}. 
It is also emphasized that Theorem \ref{t3} does not 
require computer calculations because these error estimates 
show that only a calculation of the first partial sum is required, 
which is easily done by hand. 

\subsection{Acknowledgements}
The author would like to thank Kazuo Akutagawa, Simon Brendle, Matt Gursky, and
Karen Uhlenbeck for enlightening discussions. 
\section{Green's function expansion for an Einstein metric}
\label{gfesec}

In this section, some general results about the mass 
of the Green's function metric of an Einstein manifold in dimension four are
presented. First, there is the following expansion for the Green's function.
\begin{proposition}
\label{greenein}
Let $(M,g)$ be an Einstein metric with positive scalar curvature
in dimension $4$, and 
Let $G$ be the Green's function for the conformal Laplacian at the
point $x_0 \in M$. Let $\{x^i\}$ be a Riemannian normal coordinate system 
at $x_0$. Then for any $\epsilon > 0$,
\begin{align}
\label{eing}
G = |x|^{-2} + A + \sum a_i x^i + O ( |x|^{2 - \epsilon})
\end{align}
as $|x| \rightarrow 0$, where $A$ and $a_i$ are constants (independent of $\epsilon$).
\end{proposition}
\begin{proof}
For any radial function $u(r)$,
\begin{align}
\Delta u = u_{rr} + \partial_r \big( \log \sqrt{\det (g) } \big) \cdot u_r.
\end{align}
Recall the expansion of the volume element 
\begin{align}
\begin{split}
\sqrt{ \det(g) } &=  1 - \frac{1}{6} R_{kl} x^k x^l - \frac{1}{12} \nabla_m R_{kl}
x^m x^k x^l + O(r^4)\\
& = 1 - \frac{R}{24} r^2 + O(r^4),
\end{split}
\end{align}
as $r \rightarrow 0$, since $g$ is assumed to be Einstein. 
Changing to radial normal coordinates, 
\begin{align}
\sqrt{ \det(g) } & = r^3 - \frac{R}{24} r^5 + O(r^4),
\end{align}
Using the expansion 
\begin{align}
\log (1 + x) = x + O(x^2)
\end{align}
as $x \rightarrow 0$, there is an expansion 
\begin{align}
 \log \sqrt{ \det(g)} = 3 \log r - \frac{R}{24} r^2 + O(r^4)
\end{align}
as $r \rightarrow 0$. Choose a non-negative cut-off function, 
$\phi(r)$ that is identically $1$ on a small ball $B(x_0,\delta)$ 
around $x_0$, and zero outside of $B(x_0,2\delta)$, with 
$\delta$ smaller than the injectivity radius at $x_0$. 
Letting $u = \phi(r) r^{-2}$, one computes that
\begin{align}
\begin{split}
\square u &= -6 ( u_{rr} + \partial_r \Big( \log \sqrt{\det (g) }\Big) \cdot u_r) + R u
+ O(1)\\
& = - 36 r^{-4} - 6\Big(  \frac{3}{r} - \frac{R}{12} r + O(r^3) \Big) (-2 r^{-3}) 
+ R r^{-2} + O(1) \\
& = O(1)
\end{split}
\end{align}
as $r \rightarrow 0$. 

Consequently, $\square u \in L^p$ for any $p > 0$. Since $\square$ is invertible,
there exists a solution of $ \square G_0 = \square u$ for $G_0 \in L^p_2$ 
for any $p  > 0 $. By the Sobolev embedding theorem, $G_0 \in C^{1,\alpha}$, 
so $G_0$ admits an expansion 
\begin{align}
G_0 = - A - \sum a_i x^i + O(r^{1 + \alpha}), 
\end{align}
which implies that 
\begin{align}
G = u - G_0 = r^{-2} + A + \sum a_i x^i + O(r^{1 + \alpha}),
\end{align}
admits the claimed expansion.
\end{proof}

The relation between the constant $A$ and the mass is 
given by the following. 
\begin{proposition}
\label{massA}
Let $(M,g)$ be Einstein, and assume that the Green's function 
admits the expansion for some $\epsilon > 0$, 
\begin{align}
\label{eing2}
G = |z|^{-2} + A + O ( |z|^{\epsilon})
\end{align}
as $|z| \rightarrow 0$, where $A$ is a constant,
and $\{z^i\}$ are Riemannian normal coordinates centered at $z_0$. 
Then the constant $A$ is related to the mass by 
\begin{align}
\mathrm{mass}(G^2 \cdot g) = 12 A - \frac{R}{12}.
\end{align}
\end{proposition}
\begin{proof}
In Riemannian normal coordinates, the metric admits the expansion
\begin{align}
g_{ij} = \delta_{ij} - \frac{1}{3} R_{ikjl}(z_0) z^k z^l + O(|z|^3)_{ij}
\end{align}
as $|z| \rightarrow 0$. Let $\{ x^i = z^i/|z|^2 \}$ denote
inverted normal coordinates near $z_0$,and let
\begin{align}
\mathcal{I}(x) = \frac{x}{|x|^2} = z
\end{align}
denote the inversion map. With respect to these
coordinates, the metric $G^2 g$ in the
complement of a large ball may be written as
\begin{align}
\begin{split}
G^2 g &= \mathcal{I}^* ( G^2 g) \\
& =  (G \circ \mathcal{I})^2 \mathcal{I}^* \Big( \{  \delta_{ij} - \frac{1}{3} R_{ikjl}(z_0) z^k z^l + O(|z|^3)_{ij} \} dz^i dz^j \Big)\\
& = ( |x|^2 + A + O(|x|^{-\epsilon}))^2
\big\{  \delta_{ij} - \frac{1}{3} R_{ikjl}(z_0) \frac{x^k x^l}{|x|^4} + O(|x|^{-3})_{ij}\big\}\\
&\cdot \frac{1}{|x|^2} \Big( \delta_{ip} - \frac{2}{|x|^2} x^i x^p \Big) dx^p
\cdot \frac{1}{|x|^2} \Big( \delta_{jq} - \frac{2}{|x|^2} x^j x^q \Big) dx^q,
\end{split}
\end{align}
so there is an expansion
\begin{align}
\begin{split}
\label{burnsexp}
(G^2 g)_{ij}(x) &= \delta_{ij} - \frac{1}{3} R_{ikjl}(z_0) \frac{x^k x^l}{|x|^4} + 2 A
\frac{1}{|x|^2} \delta_{ij} + O( |x|^{-2 - \epsilon})
\end{split}
\end{align}
as $|x| \rightarrow \infty$. Clearly, $g_N$ is asymptotically flat (AF) of order $\gamma = 2$, so the mass is well-defined and independent of the 
coordinate system \cite{Bartnik}. 

Since $g$ is Einstein, 
\begin{align}
R_{ikjl} = W_{ikjl} + \frac{R}{12} ( g_{ij} g_{kl} - g_{kj}g_{il}).
\end{align}
The quadratic term in \eqref{burnsexp} is then
\begin{align}
\begin{split}
- \frac{1}{3} R_{ikjl}(z_0) \frac{x^k x^l}{|x|^4} + 2 A
\frac{1}{|x|^2} \delta_{ij}
&= - \frac{1}{3}   W_{ikjl}(z_0) \frac{x^k x^l}{|x|^4}
+ \frac{1}{36} \frac{x^i x^j}{|x|^4} 
+ \Big( 2 A - \frac{R}{36} \Big)  \frac{1}{|x|^2} \delta_{ij}.
\end{split}
\end{align}
This implies the expansion
\begin{align}
\begin{split}
\label{burnsexp2}
(G^2 g)_{ij}(x) &= \delta_{ij} - \frac{1}{3}   W_{ikjl}(z_0) \frac{x^k x^l}{|x|^4}
+ \frac{R}{36} \frac{x^i x^j}{|x|^4} 
+ \Big( 2 A - \frac{R}{36} \Big)  \frac{1}{|x|^2} \delta_{ij}  + O( |x|^{-2 - \epsilon}),
\end{split}
\end{align}
as $|x| \rightarrow \infty$. 

Note that 
\begin{align}
\partial_i \Big(   W_{ikjl}(z_0) \frac{x^k x^l}{|x|^4} \Big) - \partial_j  \Big(W_{ikil}(z_0) \frac{x^k x^l}{|x|^4} \Big) = 0,
\end{align}
due to the symmetries of the Weyl tensor. Also,
\begin{align}
\partial_i \Big( \frac{x^i x^j}{|x|^4} \Big) - \partial_j \Big( \frac{x^i x^i}{|x|^4} \Big)
= 3 \frac{x^j}{|x|^4}, 
\end{align}
and
\begin{align}
\partial_i \Big(  \frac{1}{|x|^2} \delta_{ij} \Big) - \partial_j \Big(  \frac{1}{|x|^2} \delta_{ii}  \Big)
= 6 \frac{x^j}{|x|^4}.
\end{align}
Consequently, using \eqref{massdef}, the mass is
\begin{align}
\begin{split}
{\mathrm{mass}}(G^2 g) &= \lim_{R \rightarrow \infty} \frac{1}{\omega_3} 
\int_{S(R)} \Big( 12 A - \frac{R}{12} \Big) \frac{x^j}{|x|^4}   ( \partial_j \ \lrcorner \ dV)\\
& =  \Big( 12 A - \frac{R}{12} \Big) \lim_{R \rightarrow \infty} \frac{1}{\omega_3} 
\int_{S(R)} \frac{1}{|x|^3} d\sigma \\
& =  12 A - \frac{R}{12}.
\end{split}
\end{align}
\end{proof}

\section{The product metric on $S^2 \times S^2$}

The main example considered is $S^2 \times S^2$ with metric
$g = g_{S^2} \times  g_{S^2}$ the product of metrics of
constant Gaussian curvature $1$. This is invariant under 
the torus action which consists of the product of counter-clockwise
$S^1$-rotations fixing the north and south poles. This action has
$4$ fixed points $(n, n), (n,s), (s, n)$, and
$(s, s)$, where $n$ and $s$ are the north and south
poles, respectively.

  Choose a normal coordinate system $(x_1,x_2)$ on the first
factor based at $n$ and another $(x_3,x_4)$ on the second factor based
at $n$. Then $(x_1, x_2, x_3, x_4)$ is 
a normal coordinate system  on the product based at $(n,n)$. 
Note that $r_1 = \sqrt{x_1^2 + x_2^2}$, is the distance function to $n$ 
on the first factor, $r_2 = \sqrt{x_3^2 + x_4^2}$ 
is the distance function to $n$ 
on the second factor, and $r = \sqrt{r_1^2 + r_2^2}$ is the distance 
function to $(n,n)$ on the product.  
Consider also the coordinate system 
$(r_1, \theta_1, r_2, \theta_2)$ so that
\begin{align}
g_{S^2 \times S^2} = dr_1^2 + \sin^2(r_1) d\theta_1^2 + dr_2^2 + \sin^2(r_2) d\theta_2^2.
\end{align}

Later, it will be convenient to make the change of variables
\begin{align}
\label{cov0}
x = \frac{1}{2}( 1 - \cos(r_1)), \ y = \frac{1}{2}(1 - \cos(r_2)).
\end{align}
Note that $0 \leq r_1 \leq \pi$ and $0 \leq r_1 \leq \pi$,
implies that $0 \leq x \leq 1$ and $0 \leq y \leq 1$. 
In this coordinate system $(n,n) = (0,0)$,
$(s,n) = (1,0)$, $(s,s) = (1,1)$, and $(n,s) = (0,1)$. 

In addition to toric invariance, this metric is also invariant under the 
diagonal symmetry:
\begin{align}
(x,y) \mapsto (y,x).
\end{align}
The symmetries
\begin{align}
(x,y) \mapsto (1-x,y)\\
(x,y) \mapsto (x, 1-y) 
\end{align}
correspond to the antipodal maps on each factor. 

\begin{figure}[t]
\includegraphics{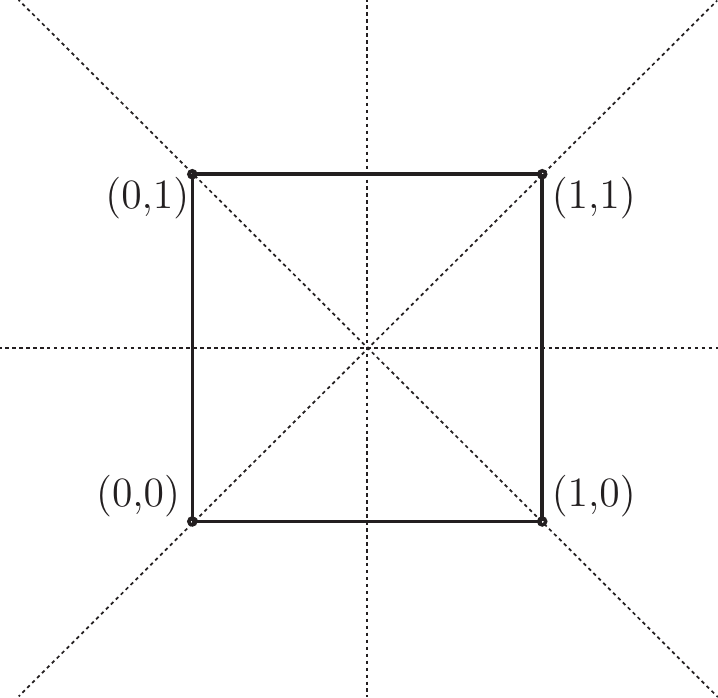}
\caption{Illustration of symmetries of $S^2 \times S^2$. 
The diagonal symmetry is a reflection in the dotted diagonal line passing
through $(0,0)$. 
Reflection in the dotted vertical line is the antipodal map
of the first factor, while reflection in the dotted horizontal line
is the antipodal map of the second factor.
}
\label{s2s2fig}
\end{figure}

\section{Green's function for $S^2 \times S^2$}
Consider the Green's function of the product metric based at $(n,n)$. 
From uniqueness of the Green's function, it must be invariant under 
the torus action. Consequently, $G = G(r_1, r_2)$. It is furthermore 
invariant under the diagonal symmetry, that is, $G(r_1,r_2) = G(r_2,r_1)$. 
\begin{proposition}
\label{ncexp}
In Riemannian normal coordinates based at $(n,n)$, 
for any $\epsilon > 0$, 
the Green's function based at $(n,n)$ admits the expansion
\begin{align}
\label{grexp}
G =   \frac{1}{r_1^2 + r_2^2} +  A - \frac{1}{180} \frac{ r_1^2 r_2^2}{r_1^2 + r_2^2} 
+ \Big( \frac{A}{12}- \frac{1}{360} \Big) (r_1^2 + r_2^2)
+ O( r^{4- \epsilon}) 
\end{align}
for some constant $A$ as $r \rightarrow 0$, where $r^2 = r_1^2 + r_2^2$. 
 \end{proposition}
\begin{proof}
It is easy to verify that for $G= G(r_1, r_2)$, the operator $\square$ is 
the operator
\begin{align}
\label{spde}
\square G = -6\big(G_{r_1 r_1} + \cot(r_1) G_{r_1} + G_{r_2 r_2} + \cot(r_2) G_{r_2}\big) + 4 G
\end{align}
on the square $[0,\pi] \times [0, \pi]$.  Since $r = \sqrt{r_1^2 + r_2^2}$ is
the distance function from $(n,n)$, the proof of Proposition \ref{greenein} above
shows that
\begin{align}
\square ( \phi(r) r^{-2}) = O(1),
\end{align}
as $r \rightarrow 0$, where $\phi(r)$ is a cutoff function chosen
as above. The next step is to identify the leading terms of the right hand side. 
Since the result is only concerned with the order of growth at $(n,n)$, 
the cutoff function will be omitted from the following calculations.

Using the expansion 
\begin{align}
\cot(s) = \frac{1}{s} - \frac{s}{3} - \frac{s^3}{45} + O(s^5)
\end{align}
as $s \rightarrow 0$, 
it follows from \eqref{spde}, that 
\begin{align}
\square (r^{-2}) = - \frac{12}{45}  + \frac{24}{45} \frac{r_1^2 r_2^2}{r^4} + O(r^2)
\end{align}
as $r \rightarrow 0$. 

A computation (which is omitted) shows that
\begin{align}
\square \Big( \frac{r_1^2 r_2^2}{r^2} \Big) = -6 \Big( 4 - 16 \frac{r_1^2 r_2^2}{r^4}\Big) + O(r^2),
\end{align}
and
\begin{align}
\square (r^2) = - 48 + O(r^2),
\end{align}
as $r \rightarrow 0$. This implies that 
\begin{align}
\square \Big( \frac{1}{180}  \frac{r_1^2 r_2^2}{r^2} + \frac{1}{360} (r_1^2 + r_2^2)\Big)
= - \frac{12}{45}  + \frac{24}{45} \frac{r_1^2 r_2^2}{r^4} + O(r^2)
\end{align}
as $r \rightarrow 0$. 

Next, defining
\begin{align}
G_{-2} =   \frac{1}{r_1^2 + r_2^2} \Big( 1 - \frac{1}{180}r_1^2 r_2^2 \Big)
- \frac{1}{360} (r_1^2 + r_2^2),
\end{align}
the above computations show that 
\begin{align}
\square G_{-2} = O(r^2)
\end{align}
as $r \rightarrow 0$. 
Consequently, $\square G_0 \in C^{1,\alpha}$ for any $\alpha < 1$. 
Since $\square$ is invertible, one may solve 
\begin{align}
\square G_0 = \square G_{-2}
\end{align}
with $G_0 \in C^{3,\alpha}$. Another straightforward computation shows that 
this function admits an expansion 
\begin{align}
G_0 = -A - \frac{A}{12} (r_1^2 + r_2^2) + O(r^{4 - \epsilon}),
\end{align}
for some constant $A$ as $r \rightarrow 0$. 
Consequently, 
\begin{align}
G = G_{-2} - G_0
\end{align}
admits the claimed expansion. 
\end{proof}

\section{The approximate Green's function}
\label{agf}
In order to have functions defined globally on $S^2 \times S^2$, 
consider the change of variables
\begin{align}
\label{cov}
x = \frac{1}{2}( 1 - \cos(r_1)), \ y = \frac{1}{2}(1 - \cos(r_2)).
\end{align}
Note that $0 \leq r_1 \leq \pi$ and $0 \leq r_1 \leq \pi$,
imply that $0 \leq x \leq 1$ and $0 \leq y \leq 1$. 

For functions $u = u(x,y)$, the operator $\square$ takes the form
\begin{align}
\square u =-6\big\{ x(1-x) u_{xx} + (1-2 x) u_x + y (1-y) u_{yy} + (1-2y) u_y \big\}
+ 4 u,
\end{align}
which is an operator on the square $[0,1] \times [0, 1]$
\begin{proposition}In the $(x,y)$-coordinates, 
for any $\epsilon > 0$, the Green's function based at $(0,0)$ 
admits the expansion
\begin{align}
G = \frac{1}{4} \frac{1}{x+y} + \frac{1}{6} \frac{xy}{(x+y)^2} 
+ A_1 + \frac{1}{9} \frac{x^2 y^2}{(x+y)^3} + \frac{A_1}{3} (x + y) 
+ O ( (x^2 + y^2)^{1 - \epsilon}),
\end{align}
for some constant $A_1$ as $(x,y) \rightarrow (0,0)$. The constants are related by 
\begin{align}
A = A_1 + \frac{1}{12}.
\end{align}
\end{proposition}
\begin{proof}
From \eqref{cov},
\begin{align}
r_1^2 = (\cos^{-1}(1 - 2x))^2 = 4x + \frac{4}{3} x^2 + \frac{32}{45} x^3 + O(x^4),
\end{align}
as $x \rightarrow 0$, and similarly
\begin{align}
r_2^2 = (\cos^{-1}(1 - 2y))^2 = 4y + \frac{4}{3} y^2 + \frac{32}{45} y^3 + O(y^4),
\end{align}
as $y \rightarrow 0$. Then expand 
\begin{align}
r_1^2 + r_2^2 = 4(x+y) + \frac{4}{3} (x^2 + y^2) +  \frac{32}{45}( x^3 + y^3) 
+ O(x^4 + y^4)
\end{align}
as $x,y \rightarrow 0$. 
This implies that 
\begin{align}
\frac{1}{ r_1^2 + r_2^2} = \frac{1}{4(x+y)} \Big\{
1 - \frac{x^2 + y^2}{3(x+y)} - \frac{8}{45} \frac{x^3 + y^3}{x+y} 
+ \Big(\frac{x^2 + y^2}{3(x+y)}\Big)^2 + O(x^4 + y^4)\Big\}
\end{align}
as $x,y \rightarrow 0$. 

From \eqref{grexp}, expand 
\begin{align}
\begin{split}\label{ghor}
G &=   \frac{1}{4(x+y)} \Big\{
1 - \frac{x^2 + y^2}{3(x+y)} - \frac{8}{45} \frac{x^3 + y^3}{x+y} 
+ \Big(\frac{x^2 + y^2}{3(x+y)}\Big)^2 \Big\}
\Big\{1 - \frac{1}{180} 16 x y \Big\}\\
&+ A + \Big( \frac{A}{12}- \frac{1}{360} \Big) 4(x + y)
+ O(( x^2 + y^2)^{1 - \epsilon}) 
\end{split}
\end{align}
as  $x,y \rightarrow 0$. The second term on the right hand side of the first line of 
\eqref{ghor} is 
\begin{align}
- \frac{ x^2 + y^2}{12(x+y)^2}  = - \frac{1}{12} + \frac{1}{6}  \frac{xy}{(x+y)^2}.
\end{align}
The remaining leading terms on the right hand side of the first line of 
\eqref{ghor} are
\begin{align}
- \frac{2}{45} \frac{x^3 + y^3}{(x+y)^2} + \frac{ (x^2 + y^2)^2}{36 (x+y)^3}
- \frac{1}{45} \frac{xy}{x+y}
= - \frac{1}{60} (x+y) + \frac{1}{9} \frac{x^2 y^2}{(x+y)^3}. 
\end{align}
Adding up all terms, 
\begin{align}
\begin{split}
G &= \frac{1}{4} \frac{1}{x+y} + \frac{1}{6} \frac{xy}{(x+y)^2} 
+ A - \frac{1}{12} + \frac{1}{9} \frac{x^2 y^2}{(x+y)^3} + \frac{1}{3} 
\Big( A - \frac{1}{12}\Big) (x + y) \\ 
&\ \ \ \ \ \ \ \ + O ( (x^2 + y^2)^{1 - \epsilon}),
\end{split}
\end{align}
as $x,y \rightarrow 0$, which finishes the computation. 
\end{proof}

Next, define
\begin{align}
\label{approxG}
\boxed{
G_{-2} \equiv \frac{1}{4} \frac{1}{x+y} + \frac{1}{6} \frac{xy}{(x+y)^2} 
+ \frac{1}{9} \frac{x^2 y^2}{(x+y)^3}}.
\end{align}
\begin{remark}{\em
This differs from the choice of $G_{-2}$ occuring in the proof of 
Propostion~\ref{ncexp} above, since the choice $A_1 = 0$ corresponds to 
choosing $A = 1/12$, but this does not matter. 
}
\end{remark}
A key formula is given by the following: 
\begin{proposition}The function $G_{-2}$ is smooth on 
$S^2 \times S^2 \setminus \{(n,n)\}$ and
\begin{align}
\square G_{-2} =   \frac{ 8 x^2 y^2 ( 5 x^2 - 8 xy + 5 y^2)}{9 (x+y)^5}.
\end{align}
\end{proposition}
\begin{proof}
This is a straightforward calculation, which is omitted. 
\end{proof}
Defining 
\begin{align}
f =  \frac{ 8x^2 y^2 ( 5 x^2 - 8 xy + 5 y^2)}{9(x+y)^5},
\end{align}
then as seen in the above proof, one must solve the 
equation
\begin{align}
\label{G0def}
\square G_0 = f, 
\end{align}
for $G_0 \in C^{3,\alpha}$.
Then, one expands  
\begin{align}
G = G_{-2} - G_0 = \frac{1}{r^2} + \frac{1}{12} - G_0(0,0) + O(r^2),
\end{align}
as $r \rightarrow 0$, so the constant $A$ is given by 
\begin{align}
A =  \frac{1}{12} - G_0(0,0),
\end{align}
and by Proposition \ref{massA}, 
\begin{align}
m_1 = \frac{2}{3} - 12 G_0(0,0).
\end{align}
To compute $G_0(0,0)$, an eigenfunction expansion will be used. 
That is, let $\phi_n$ denote a $L^2$ orthonormal basis 
of eigenfunctions, with eigenvalues $\lambda_n$, where $n$ 
corresponds to some indexing of the eigenvalues. 
The function $f$ chosen above is in $L^2$, so admits an expansion
\begin{align}
f = \sum_{n=0}^{\infty} f_n \phi_n, 
\end{align}
Then $G_0$ is given by 
\begin{align}
G_0 =  \sum_{n=0}^{\infty} \frac{f_n}{\lambda_n} \phi_n.
\end{align}
The above sum converges in $L^2$. If it happens to 
converge pointwise at $(0,0)$, then
\begin{align}
\label{massf1}
m_1 = \frac{2}{3} - 12  \sum_{n=0}^{\infty} \frac{f_n}{\lambda_n} \phi_n(0,0).
\end{align}
However, for general $f \in L^2$, there is no reason that the 
above sum should converge. Fortunately,  $f$ as defined 
above has better regularity properties, and it turns out that the sum 
does indeed coverge absolutely, this will be discussed in 
detail below. Moroever, a rigorous estimate for the rate of 
convergence of the partial sums will be given in Section \ref{Error}. 

\section{The mass of $S^2 \times S^2$}
The eigenvalues of  $-\Delta$ on $S^2$ are 
$\lambda_j = j(j+1)$, $j \geq 0$ with associated 
eigenspace of dimension $2j +1$. Only 
functions invariant under rotations fixing the north and 
south poles will be considered, and there is a unique invariant 
eigenfunction under this action, up to scaling. 
Choosing normal coordinates based at the north pole in which 
\begin{align}
g_{S^2} = dr^2 + \sin^2(r) d\theta^2,
\end{align}
the invariant eigenfunction is given by 
\begin{align}
\label{phik}
\phi_j = \sqrt{ \frac{ 2j + 1}{4 \pi}} P_j ( \cos(r)). 
\end{align}
where $P_j$ is the Legendre polynomial of degree $j$. 
The choice
of normalization in \eqref{phik} is to ensure that
$\Vert \phi_j \Vert_{L^2(S^2)} = 1$.  

The invariant eigenfunction is, by definition, 
a solution of the ODE
\begin{align}
\label{efe}
u''(r) + \cot(r) u'(r) + j(j + 1) u(r) = 0.
\end{align}
This has a $2$-dimensional space of solutions, but only
a $1$-dimensional subspace yields smooth solutions
defined on $S^2$. 
Making a change of variables $u(r) = v( \cos(r))$, and 
letting $x = \cos(r)$, \eqref{efe} transforms into
\begin{align}
(1 - x^2) v''(x) - 2 x v'(x) + j(j+1) v(x) = 0,
\end{align}
which is the more familiar Legendre equation, 
which has a unique polynomial solution, $P_j(x)$,
normalized so that $P_j(1) = 1$.

From \cite{Bergeretal}, the eigenvalues on $S^2 \times S^2$
are given by 
\begin{align}
j(j+1) + k(k+1), \ j, k \geq 0,
\end{align}
with eigenfunctions the product of eigenfunctions on each factor. 
Taking into account the torus actions, only the invariant 
eigenfunctions 
\begin{align}
\phi_{j,k} = \phi_j (r_1) \phi_k(r_2),  \ j, k \geq 0. 
\end{align}
need be considered.
By Fubini's Theorem, these satisfy $\Vert \phi_{j,k} \Vert_{L^2(S^2 \times S^2)} = 1$. 

If $\lambda$ is an eigenvalue of $-\Delta$, then 
$6 \lambda + 4$ is an eigenvalue of $\square$,
so the eigenvalues of $\square$ are given by 
\begin{align}
\lambda_{j,k} \equiv 6( j(j+1) + k(k+1)) + 4, \ j, k \geq 0.
\end{align}

From \eqref{massf1}, it follows that
\begin{align}
m_1 = \frac{2}{3}  - \frac{3}{ \pi}\sum_{j,k=0}^{\infty} \frac{\sqrt{2j+1} \sqrt{2k+1} }{6( j(j+1) + k(k+1)) + 4} f_{j,k},
\end{align}
where 
\begin{align}
f_{j,k} = \int_{S^2 \times S^2} f \phi_j \phi_k dV_g.
\end{align}

\subsection{Evaluation of Fourier coefficients}
\label{efc}
Since $f$ is a function of $r_1$ and $r_2$, 
\begin{align}
\begin{split}
f_{j,k} &= \int_{S^2 \times S^2} f  \phi_{j,k} dV_g\\
& = (4 \pi^2) \int_0^{\pi} \int_0^{\pi} 
f(r_1,r_2) \phi_j(\cos(r_1)) \phi_k( \cos(r_2)) \sin(r_1)
\sin(r_2) dr_1 dr_2.
\end{split}
\end{align}
Substituting the change of variables \eqref{cov}, 
noting that
\begin{align}
\sin(r_1) dr_1 = 2 x dx,\ \sin(r_2) dr_2 = 2 y dy,
\end{align} 
this becomes
\begin{align}
\begin{split}
f_{j,k} & = 16 \pi^2 \int_0^1 \int_0^1 f(x,y) \phi_j(1-2x) \phi_k(1-2y) dx dy\\
&= 16 \pi^2 \int_0^1 \int_0^1  \frac{8 x^2 y^2 ( 5 x^2 - 8 xy + 5 y^2)}{9(x+y)^5}
\phi_j(1-2x) \phi_k(1-2y) dx dy.
\end{split}
\end{align}
The key point is that this is now a rational integral 
which can be evaluated explicitly. 

Next, use the following expansion for the Legendre polynomials:
\begin{proposition}
The Legendre polynomials have the explicit expression
\begin{align}
P_j(1 - 2x)  = \sum_{p = 0}^{j} {j \choose p} {j + p \choose p}(-x)^p.
\end{align}
Furthermore, 
\begin{align}
P_j(1) &= 1, \\
P_j(-1) &= (-1)^j. 
\end{align}
\end{proposition}
\begin{proof}
This is classical; see for example \cite[Chapter 22]{Abram}.
\end{proof}
Next, write
\begin{align}
\begin{split}
f_{j,k} & = 4 \pi \sqrt{2j+1}\sqrt{2k+1}  
\sum_{p = 0}^{j}\sum_{q = 0}^{k}  {j \choose p} {j + p \choose p}
 {k \choose q} {k + q \choose q}(-1)^{p+q} \tilde{f}_{p,q},
\end{split}
\end{align}
where
\begin{align}
\tilde{f}_{p.q} = \int_0^1 \int_0^1 f(x,y) x^p x^q dx dy.
\end{align}
\begin{proposition}
\label{hardprop}
Letting $A(p)$ be the partial sum of the alternating 
harmonic series, 
\begin{align}
A(p) = \sum_{i=1}^p (-1)^{i-1} \frac{1}{i}, 
\end{align}
with $A(0) \equiv 0$, then 
\begin{align}
\begin{split}
\label{tfpq}
\tilde{f}_{p.q} &= \frac{1}{18} \frac{1}{p + q +3} \Big\{ 
-109 - 68 (p+q) - 33 (p^2+ q^2)  - 6 (p^3 + q^3)  \\
&+ 4(-1)^{p}  (p+1) (p+2) (3p^2 + 9p +10)\big( \log(2) - A(p)\big)\\
&+4(-1)^{q} (q+1) (q+2) (3q^2 + 9q +10)\big( \log(2) - A(q)\big) \Big\}.
\end{split}
\end{align}
\end{proposition}
\begin{proof}
The proof can be found in Appendix \ref{app}. 
\end{proof}
The explicit formula for the mass is given by:
\begin{theorem}
\label{masss2s2}
The mass of the Green's function metric 
of the product metric on 
$S^2 \times S^2$, denoted by $m_1$, 
is given by 
\begin{align}
\label{ms2form}
m_1=  
\frac{2}{3} - 12 \sum_{j,k=0}^{\infty} \frac{(2j+1)(2k+1) }{\lambda_{j,k}} 
 \Big( \sum_{p = 0}^{j}\sum_{q = 0}^{k} c_{j,k}^{p,q} \tilde{f}_{p,q} \Big),
\end{align}
where
\begin{align}
\lambda_{j,k} &= 6\big( j(j+1) + k(k+1)\big) + 4,\\
c_{j,k}^{p,q} & = (-1)^{p+q}{j \choose p} {j + p \choose p}
{k \choose q} {k + q \choose q},
\end{align}
and $\tilde{f}_{p.q}$ is defined in \eqref{tfpq}.
\end{theorem}
\begin{proof}
This follows directly from the above calculations. 
\end{proof}
\subsection{Convergence properties}
\label{Error}
For any $h \in L^2(S^2 \times S^2)$, let
\begin{align}
S_N h = \sum_{j = 0}^N \sum_{k=0}^N h_{j,k} \phi_{j,k}
\end{align}
be the $N$th square partial sum, where 
\begin{align}
h_{j,k} = \int_{S^2 \times S^2} h \phi_{j,k} dV.
\end{align}
If $h(x,y) = h(y,x)$, then $h_{j,k} = h_{k,j}$. 
With these assumptions, estimate
\begin{align}
\begin{split}
|h - S_N h| &= \Big| \Big( \sum_{j = N+1}^{\infty} \sum_{k = N+1}^{\infty}
+ \sum_{j = 0}^N \sum_{k = N+1}^{\infty} +   \sum_{j = N+1}^{\infty}\sum_{k = 0}^N \Big)
 h_{j,k} \phi_{j,k} \Big|\\
& = \Big| \Big( \sum_{j = N+1}^{\infty} \sum_{k = N+1}^{\infty}
+ 2 \sum_{j = 0}^N \sum_{k = N+1}^{\infty} \Big) h_{j,k} \phi_{j,k} \Big|.
\end{split}
\end{align} 
Next add the assumption that $h \in L^2_4$, then 
\begin{align}
(\square^2 h)_{j,k} = \lambda_{j,k}^2 h_{j,k},
\end{align}
and the above estimate becomes
\begin{align}
\begin{split}
|h - S_N h| 
& = \Big( \sum_{j = N+1}^{\infty} \sum_{k = N+1}^{\infty}
+ 2 \sum_{j = 0}^N \sum_{k = N+1}^{\infty} \Big) \Big|\frac{1}{\lambda_{j,k}^2} (\square^2 h)_{j,k} \phi_{j,k} \Big| \\
& \leq \Big\{ \Big(\sum_{j = N+1}^{\infty} \sum_{k = N+1}^{\infty}
+ 2 \sum_{j = 0}^N \sum_{k = N+1}^{\infty}  \Big) \frac{\phi^2_{j,k}}{\lambda_{j,k}^4} \Big\}^{1/2}
\Vert \square^2 h \Vert_{L^2}.
\end{split}
\end{align} 
Note  that
\begin{align}
\frac{\phi^2_{j,k}}{\lambda_{j,k}^4}
\leq \frac{1}{16 \pi^2} \frac{ (2j+1) (2k+1)}{ (6 j(j+1) + 6k(k+1) +4)^4}.
\end{align}
The right hand side is a convex function of $j$ and~$k$,
strictly decreasing in both $j$ and $k$. Consequently, 
by the integral test, the following estimates hold 
\begin{align}
\begin{split}
\label{int1}
\sum_{j = N+1}^{\infty} \sum_{k = N+1}^{\infty}  \frac{\phi^2_{j,k}}{\lambda_{j,k}^4}
&\leq  \frac{1}{16 \pi^2} \Bigg( 
\int_{j = N}^{\infty} \int_{k = N}^{\infty} 
\frac{  (2j+1) (2k+1)}{ (6 j(j+1) + 6k(k+1) +4)^4} dj dk \Bigg),\\
\end{split}
\end{align}
and
\begin{align}
\begin{split}
\label{int2}
\sum_{j = 0}^N \sum_{k = N+1}^{\infty}  \frac{\phi^2_{j,k}}{\lambda_{j,k}^4}
&\leq  \frac{1}{16 \pi^2} \Bigg(  \int_{j = 0}^{N} \int_{k = N}^{\infty}    
\frac{ (2j+1) (2k+1)}{ (6 j(j+1) + 6k(k+1) +4)^4} dj dk \\
&\ \ \ \ \ + \int_{N}^{\infty} \frac{ 2k+1}{ (6k(k+1) +4)^4} dk \Bigg).
\end{split}
\end{align}
The integral in \eqref{int1} is given by 
\begin{align}
\begin{split}
\int_{j = N}^{\infty} \int_{k = N}^{\infty} 
\frac{  (2j+1) (2k+1)}{ (6 j(j+1) + 6k(k+1) +4)^4} dj dk
= \frac{1}{3456} \frac{1}{  (3N^2 + 3N +1 )^2}.
\end{split}
\end{align}
The first integral in \eqref{int2} is given by 
\begin{align}
\begin{split}
\int_{j = 0}^{N} \int_{k = N}^{\infty}    
& \frac{ (2j+1) (2k+1)}{ (6 j(j+1) + 6k(k+1) +4)^4} dj dk \\
& =  \frac{1}{864} \Big( 
\frac{1}{ (3N^2 + 3N + 2)^2} \Big)
- \frac{1}{3456} \Big( \frac{1}{(3N^2 + 3N +1)^2} \Big).
\end{split}
\end{align}
The second integral in \eqref{int2} is given by 
\begin{align}
\begin{split}
  \int_{N}^{\infty} \frac{ 2k+1}{ (6k(k+1) +4)^4} dk
= \frac{1}{144} \frac{1}{(3 N^2 + 3N +2)^3}.
\end{split}
\end{align}
It follows that
\begin{align}
\begin{split}
|h - S_N h| \leq  \frac{1}{48 \pi} F(N) \Vert \square^2 h \Vert_{L^2}.
\end{split}
\end{align}
where 
\begin{align}
\label{Fdef}
F(N) = \Bigg\{  - \frac{1}{24}
\frac{1}{  (3N^2 + 3N +1 )^2}
+ \frac{1}{3} \frac{1}{ (3N^2 + 3N +2)^2} + \frac{1}{(3N^2 + 3N +2)^3} \Bigg\}^{1/2}.
\end{align} 
Now choose $h = G_0$, defined in \eqref{G0def}, then
\begin{align}
\Vert \square^2 G_0 \Vert_{L^2} = \Vert \square f \Vert_{L^2} =
4 \pi \Bigg\{ \int_0^1 \int_0^1 (\square f)^2 dx dy \Bigg\}^{1/2}.
\end{align}
The latter is rational integral which can be explicity 
computed:
\begin{align}
\int_0^1 \int_0^1 (\square f)^2 dx dy = \frac{61547}{45045},
\end{align}
(the computation is a straghtforward extension of the 
arguments in Appendix \ref{app}, and is omitted). 
\begin{theorem}
For $G_0$ defined in \eqref{G0def}, 
\begin{align}
\label{error}
\begin{split}
|G_0 - S_N G_0| \leq  \frac{1}{12} \sqrt{\frac{61547}{45045}} \cdot F(N)
\end{split}
\end{align}
where $F(N)$ is defined in \eqref{Fdef}. 
\end{theorem}
\subsection{Estimates on the mass}
The following theorem is computable by hand, 
and does not depend upon computer calculations:
\begin{theorem}
\label{numprop0}
The mass $m_1$ satisfies the estimates
\begin{align}
\label{fiest}
.4946 < m_1 < .6803.
\end{align}
\end{theorem}
\begin{proof} 
It is straightforward to verify that the square
partial sum in \eqref{ms2form} corresponding to $N=1$ is
\begin{align}
S_1 \equiv \frac{2}{3} - \frac{12}{56} \Big( 
53\tilde{f}_{0,0} 
-57\tilde{f}_{1,0} 
- 57\tilde{f}_{0,1} 
+ 72 \tilde{f}_{1,1} \Big).
\end{align} 
Using \eqref{tfpq}, this is equal to 
\begin{align}
S_1 = \frac{4777}{1260} - \frac{208 \log(2)}{45} \sim .5873...
\end{align}
The error estimate \eqref{error} yields 
\begin{align}
|m_1 - S_1 | \leq  \sqrt{\frac{61547}{45045}} \cdot F(1) \sim .09286...,
\end{align}
which obviously implies the stated estimates. 
\end{proof}

This implies the following corollary, which covers the
first case in Theorem \ref{t3}:
\begin{corollary} The value of $- 9(m_1)^{-1}$ is different from 
$-1/3$. 
\end{corollary}
\begin{proof}
The estimate \eqref{fiest} implies that 
\begin{align}
-.2247 < - (9m_1)^{-1} < -.1632,
\end{align}
so cannot be equal to $-1/3 \sim -.3333...$.
\end{proof}
To obtain a better estimate on the mass, one must use a computer to 
calculate more terms in the sum. To this end:
\begin{proposition}
\label{numprop}
The mass $m_1$ satisfies the estimates
\begin{align}
.58722 < m_1 < .58727.
\end{align}
\end{proposition}
\begin{proof} 
Using Mathematica, the square partial sum in \eqref{ms2form} 
corresponding to $N=100$ is equal to 
\begin{align}
S_{100} \sim .5872473203....
\end{align}
The error estimate \eqref{error} yields 
\begin{align}
|m_1 - S_{100} | <.0000209,
\end{align}
which implies the stated estimates.
\end{proof}

This also yields a better estimate on $t_0 = -(9m_1)^{-1}$ appearing 
in Theorem \ref{t2}:
\begin{corollary}
The value of $- (9m_1)^{-1}$ satisfies
\begin{align}
-.1892 < - (9m_1)^{-1} < -.18922,
\end{align}
and the value of $- 2 (9m_1)^{-1}$ satisfies
\begin{align}
-.3784 < -2 (9m_1)^{-1} < - .37842 . 
\end{align}
\end{corollary}

\section{The mass of $G(2,4)$}
\label{G24sec}
For a Riemannian manifold $(M,g)$, denote by $G_{M,p}$ the Green's function 
of $(M,g)$ with pole at $p$.
As mentioned in the introduction, the product metric on $S^2 \times S^2$ 
admits the Einstein quotient $G(2,4)$,
where $\ZZ_2$ acts as the product of the antipodal map on both factors.
Since the conformal Laplacian is linear, and the Green's
function is unique, it follows that
\begin{align}
\label{pig}
\pi^* (G_{G(2,4), \pi(0,0)}) = G_{S^2 \times S^2, (0,0)} +  G_{S^2 \times S^2, (1,1)},
\end{align}
where $\pi: S^2 \times S^2 \rightarrow G(2,4)$ is the projection map. 

This function in \eqref{pig} has the expansion
\begin{align}
\pi^* (G_{G(2,4), \pi(0,0)}) =  \frac{1}{r^2} + A_{G(2,4)} + O(r^2),
\end{align}
and the constant term is given by 
\begin{align}
 A_{G(2,4)} = A_{S^2 \times S^2} +   G_{S^2 \times S^2, (1,1)} (0,0).
\end{align}
But from uniqueness of the Green's function, and since
the quotient symmetry is $(x,y) \mapsto (1-x,1-y)$, 
\begin{align}
 G_{S^2 \times S^2, (1,1)} (x,y) =  G_{S^2 \times S^2, (0,0)} (1-x,1-y).
\end{align}
It follows that
\begin{align}
 A_{G(2,4)} = A_{S^2 \times S^2} +   G_{S^2 \times S^2, (0,0)} (1,1).
\end{align}
The first term on the right hand side was computed above. 
To compute the second term, it was shown above in Section \ref{agf} 
that 
\begin{align}
G_{S^2 \times S^2, (0,0)} = G_{-2} - G_0, 
\end{align}
where $\square G_0 = \square G_{-2}$. 
Next, evaluate  
\begin{align}
\begin{split}
G_{(0,0)} (1,1) &= G_{-2} (1,1) - G_0 (1,1)\\
& = \frac{1}{4} \cdot \frac{1}{2} + \frac{1}{6} \cdot \frac{1}{4} 
+ \frac{1}{9} \cdot\frac{1}{8} 
- \sum_{j,k}  \frac{1}{ \lambda_{j,k}} f_{j,k} \phi_j(1) \phi_k(1)\\
& = \frac{13}{72} 
- \frac{1}{4\pi} \sum_{j,k} \frac{\sqrt{(2j+1)(2k+1)}}{ \lambda_{j,k}} (-1)^{j+k} f_{j,k}.
\end{split}
\end{align}
Consequently, 
\begin{theorem}
\label{massg24}
The mass of $G(2,4)$, denoted by $m_2$, is given by 
\begin{align}
m_2 =  
\frac{17}{6} - 12 \sum_{j,k=0}^{\infty} \frac{(2j+1)(2k+1) }{\lambda_{j,k}} 
(1 + (-1)^{j+k}) \Big( \sum_{p = 0}^{j}\sum_{q = 0}^{k} c_{j,k}^{p,q} \tilde{f}_{p,q} \Big).
\end{align}
\end{theorem}
\begin{proof}
Combining the above formulas,
\begin{align}
\begin{split}
m_2 &= 12 A_{G(2,4)} - \frac{R}{12} \\
& = 12 \big( A_{S^2 \times S^2} +   G_{S^2 \times S^2, (0,0)} (1,1) \big) - \frac{1}{3}\\
& = \frac{2}{3} - \frac{3}{\pi}\sum_{j,k} 
\frac{\sqrt{2j+1} \sqrt{2k+1} }{ \lambda_{j,k}} f_{j,k}\\
& + 12 \Big(
\frac{13}{72} 
- \frac{1}{4\pi} \sum_{j,k} \frac{\sqrt{(2j+1)(2k+1)}}{ \lambda_{j,k}} (-1)^{j+k} f_{j,k} \Big),
\end{split}
\end{align}
which yields the formula.
\end{proof}
\subsection{Numerical evaluation}
The error estimate is carried out exactly as in Section \ref{Error}, 
the details are omitted. Only the following result is stated here:
\begin{proposition}
The mass of $G(2,4)$ is approximately
\begin{align}
m_2 \sim 2.6289.
\end{align}
\end{proposition}
\begin{proof}
The details are analogous to that of Proposition \ref{numprop}, and are omitted. 
\end{proof}
This section is concluded by noting that 
\begin{align}
(-9m_2)^{-1} &\sim -.0422,\\
2 (-9m_2)^{-1} &\sim -.0845.
\end{align}
The first is not equal to $-1/3$, which proves this case of Theorem \ref{t3}. 

\section{The mass of $\RP^2 \times \RP^2$}
\label{rp2rp2sec}
As mentioned in the introduction, the product metric on $S^2 \times S^2$ 
admits the Einstein quotient $\RP^2 \times \RP^2$,
where $\ZZ_2 \oplus \ZZ_2$ acts by the antipodal map on each factor.
Again, since the conformal Laplacian is linear, and the Green's
function is unique, it follows that
\begin{align}
\label{pig2}
\pi^* (G_{\RP^2 \times \RP^2, \pi (0,0)}) = G_{S^2 \times S^2, (0,0)} + G_{S^2 \times S^2, (1,0)} 
+  G_{S^2 \times S^2, (0,1)}  + G_{S^2 \times S^2, (1,1)},
\end{align}
where $\pi: S^2 \times S^2 \rightarrow \RP^2 \times \RP^2$ is the projection map. 

This function in \eqref{pig2} has the expansion
\begin{align}
\pi^* (G_{\RP^2 \times \RP^2, \pi (0,0)}) =  \frac{1}{r^2} + A_{\RP^2 \times \RP^2} + O(r^2),
\end{align}
and the constant term is given by 
\begin{align}
 A_{\RP^2 \times \RP^2} = A_{S^2 \times S^2} +  G_{S^2 \times S^2, (1,0)} (0,0)+  G_{S^2 \times S^2, (0,1)} (0,0) +  G_{S^2 \times S^2, (1,1)} (0,0).
\end{align}
Again, from uniqueness of the Green's function, and since
the quotient map is generated by 
$(x,y) \mapsto (1-x,y)$ and $(x,y) \mapsto (x,1-y)$, 
\begin{align}
 G_{S^2 \times S^2, (1,0)} (x,y) =  G_{S^2 \times S^2, (0,0)} (1-x, y),
\end{align}
and a similar argument yields that 
\begin{align}
 G_{S^2 \times S^2, (0,1)} (x,y) =  G_{S^2 \times S^2, (0,0)} (x, 1- y).
\end{align}
It follows that
\begin{align}
 A_{\RP^2 \times \RP^2} = A_{S^2 \times S^2} +   G_{S^2 \times S^2, (0,0)} (1,0) +   G_{S^2 \times S^2, (0,0)} (0,1)+  G_{S^2 \times S^2, (0,0)} (1,1).
\end{align}
The first and last term on the right hand side were computed above in 
Section \ref{G24sec}. 
To compute the other terms, recall that
\begin{align}
G_{(0,0)} = G_{-2} - G_0, 
\end{align}
where $\square G_0 = \square G_{-2}$. 
Next, evaluate  
\begin{align}
\begin{split}
G_{(0,0)} (1,0) &= G_{-2} (1,0) - G_0 (1,0)\\
& = \frac{1}{4} - \sum_{j,k}  \frac{1}{ \lambda_{j,k}} f_{j,k} \phi_j(1) \phi_k(1)\\
& = \frac{1}{4} 
- \frac{1}{4\pi} \sum_{j,k} \frac{\sqrt{(2j+1)(2k+1)}}{ \lambda_{j,k}} (-1)^{j} f_{j,k}. 
\end{split}
\end{align}
Consequently,
\begin{theorem}
\label{massrp2rp2}
The mass of $ \RP^2 \times \RP^2$, denoted by $m_3$, is given by 
\begin{align}
m_3 =  
\frac{53}{6} - 12 \sum_{j,k=0}^{\infty} \frac{(2j+1)(2k+1) }{\lambda_{j,k}} 
\big(1 + (-1)^j + (-1)^k+ (-1)^{j+k} \big) 
\Big( \sum_{p = 0}^{j}\sum_{q = 0}^{k} c_{j,k}^{p,q} \tilde{f}_{p,q} \Big).
\end{align}
\end{theorem}
\begin{proof}
Using the above formulas, compute that
\begin{align}
\begin{split}
m_2 &= 12 A_{   \RP^2 \times \RP^2 } - \frac{R}{12} \\
& = 12 \big( A_{S^2 \times S^2} +   2  G_{S^2 \times S^2, (0,0)} (1,0)+  G_{S^2 \times S^2, (0,0)} (1,1) \big) - \frac{1}{3}\\
& = \frac{2}{3} - \frac{3}{\pi}\sum_{j,k} 
\frac{\sqrt{2j+1} \sqrt{2k+1} }{ \lambda_{j,k}} f_{j,k}\\
& +  12 \Big(
\frac{1}{4} 
- \frac{1}{4\pi} \sum_{j,k} \frac{\sqrt{(2j+1)(2k+1)}}{ \lambda_{j,k}}  (-1)^{j} f_{j,k} \Big)\\
& +  12 \Big(
\frac{1}{4} 
- \frac{1}{4\pi} \sum_{j,k} \frac{\sqrt{(2j+1)(2k+1)}}{ \lambda_{j,k}}  (-1)^{k} f_{j,k} \Big)\\
& + 12 \Big(
\frac{13}{72} 
- \frac{1}{4\pi} \sum_{j,k} \frac{\sqrt{(2j+1)(2k+1)}}{ \lambda_{j,k}} (-1)^{j+k} f_{j,k} \Big),
\end{split}
\end{align}
which yields the formula.
\end{proof}
\subsection{Numerical evaluation}
Again, the error estimate is carried out exactly as in Section \ref{Error}, 
the details are omitted. Only the following result is stated here:
\begin{proposition}The mass of $\RP^2 \times \RP^2$ is approximately
\begin{align}
m_2 \sim 8.4323. 
\end{align}
\end{proposition}
\begin{proof}
The details are analogous to that of Proposition \ref{numprop}, 
and are omitted. 
\end{proof}
This section is concluded by noting that 
\begin{align}
(-9m_3)^{-1} &\sim -.0131,\\
2 (-9m_3)^{-1} &\sim -.0263,
\end{align}
The first is not equal to $-1/3$, which proves this case of Theorem \ref{t3}. 
\appendix
\section{Evaluation of integrals}
\label{app}
Recalling that
\begin{align}
f =  \frac{ 8x^2 y^2 ( 5 x^2 - 8 xy + 5 y^2)}{9(x+y)^5},
\end{align}
in this appendix, the integral 
\begin{align}
\tilde{f}_{p.q} = \int_0^1 \int_0^1 f(x,y) x^p y^q dx dy
\end{align}
is computed explicitly. 
The proof will be a sequence of propositions, beginning with:
\begin{proposition}Let $p,q,$ and $n$ be integers satisfying 
$p \geq 0$, $q \geq 0$, and $n \geq 1$. Define 
\begin{align}
I(p,q;n) = \int_0^1 \int_0^y \frac{x^p y^q}{(x+y)^n} dx dy .
\end{align}
Then for $n > 1$ and $p+q > n -2$,
\begin{align}
\label{itform}
I(p,q;n) = \frac{-1}{ 2^{n-1}(n-1)(p+q+2-n)} + \frac{p}{n-1} I(p-1,q;n-1).
\end{align}
\end{proposition}
\begin{proof}
Integrating by parts,
\begin{align}
\begin{split}
I(p,q;n) &= \int_0^1 \int_0^y \frac{x^p y^q}{(x+y)^n} dx dy \\
& = \int_0^1 \int_0^y x^p y^q \frac{1}{1-n} \frac{\partial}{\partial x}
\Big( \frac{1}{(x+y)^{n-1}} \Big) dx dy\\
& = \frac{1}{1-n} \int_0^1 \Big\{ 
 x^p y^q \frac{1}{(x+y)^{n-1}}\Big|^y_0 - \int_0^y p x^{p-1} y^q  \frac{1}{(x+y)^{n-1}}
\Big\}\\
& = \frac{1}{1-n} \Big\{ \int_0^1 \frac{1}{2^{n-1}}y^{p+q - n +1} dy - p I(p-1,q;n-1) \Big\}\\
& =  \frac{1}{1-n} \Big\{\frac{1}{2^{n-1}} \frac{1}{p+q+2 -n} - p I(p-1,q;n-1) \Big\},
\end{split}
\end{align}
which yields the formula.
\end{proof}
The next proposition is:
\begin{proposition}
\label{i1prop}
For all $p \geq 0, q \geq 0$, 
\begin{align}
\label{ipq1}
I(p,q;1) 
= \frac{1}{p+q+1} (-1)^p ( \log(2) - A(p)).
\end{align}
where $A(p)$ is the $p$th partial sum of the alternating 
harmonic series, 
\begin{align}
A(p) = \sum_{i=1}^p (-1)^{i-1} \frac{1}{i}, 
\end{align}
with $A(0) \equiv 0$. 
\begin{proof}
First, if $p = 0$, compute
\begin{align}
\begin{split}
I(0,q;1) &= \int_0^1 \int_0^y \frac{y^q}{x+y} dx dy\\
& = \int_0^1 y^q  \Big( \log(x +y) \Big) \Big|^y_0 dy \\
& = \frac{1}{q+1}  \log(2),
\end{split}
\end{align}
which is indeed equal to the right hand side of \eqref{ipq1}. 

Next, if $p \geq 1$, long divide as polynomials in $x$, 
\begin{align}
\frac{x^p} {x + y} = \sum_{i = 0}^{p-1} (-1)^i x^{p-1 -i} y^{i} + (-1)^p \frac{y^p}{x +y}.
\end{align}
Using this, it follows that
\begin{align}
\begin{split}
I(p,q;1)&=  \int_0^1 \int_0^y \frac{x^p y^q}{(x+y)} dx dy \\
&= \int_0^1 y^q \int_0^y 
\Big( \sum_{i = 0}^{p-1} (-1)^i x^{p-1 -i} y^{i} + (-1)^p \frac{y^p}{x +y} \Big) dx dy\\
& = \int_0^1 y^q 
\Big( \sum_{i = 0}^{p-1} (-1)^i \frac{1}{p-i} 
x^{p-i} y^{i} + (-1)^p y^p \log(x +y) \Big)\Big|^y_0 dy\\
& =  \int_0^1 y^{p+q} \Big( \sum_{i = 0}^{p-1} (-1)^i \frac{1}{p-i} +(-1)^p \log(2) \Big) dy\\
& = \frac{1}{p+q+1}  (-1)^p ( \log(2) - A(p)).
\end{split}
\end{align}
\end{proof}
\end{proposition}
The final proposition is: 
\begin{proposition}
\label{ipqform}
For $p \geq 4$, $q \geq 0$, 
\begin{align}
\begin{split}
\label{pq5}
I(p,q;5) &= \frac{-1}{192(p+q-3)} \Big( 4 p^3 - 10 p^2 + 8p +3 \Big)\\
& \ \ + \frac{p(p-1)(p-2)(p-3)}{3 \cdot 2^3  (p+q-3)}  (-1)^p (  \log(2) - A(p -4)).
\end{split}
\end{align}
For $q > 0$,
\begin{align}
\label{3q5}
I(3,q;5) &= \frac{1}{64 q}.
\end{align}
For $q > 1$, 
\begin{align}
\label{2q5}
I(2,q;5) &= \frac{5}{192 (q-1)}.
\end{align}
\end{proposition}
\begin{proof}
If $p \geq 4$, iterating the above formula \eqref{itform} four times, it follows
that
\begin{align}
\begin{split}
I(p,q;5) &= \frac{-1}{p+q-3} \Big( \frac{1}{2^6} + \frac{p}{3 \cdot 2^5}
+ \frac{p(p-1)}{3 \cdot 2^5} + \frac{p(p-1)(p-2)}{3 \cdot 2^4} \Big)\\
&+ \frac{p(p-1)(p-2)(p-3)}{3 \cdot 2^3}I(p-4, q, 1),
\end{split}
\end{align}
which simplifies to 
\begin{align}
\begin{split}
I(p,q;5) &= \frac{-1}{192(p+q-3)} \Big( 4 p^3 - 10 p^2 + 8p +3 \Big)\\
& \ \ + \frac{p(p-1)(p-2)(p-3)}{3 \cdot 2^3} I(p-4, q, 1),
\end{split}
\end{align}
then using Proposition \ref{i1prop}, \eqref{pq5} follows. 
Next, \eqref{3q5} follows by applying \eqref{itform} three times, 
and \eqref{2q5} follows by applying \eqref{itform} two times,
the computations are straightforward and are omitted. 
\end{proof}
Now it is possible compute the required integrals. 
First, rewrite the integral on the square in terms 
of integrals on triangular regions
using the symmetry $(x,y) \mapsto (y,x)$, 
\begin{align}
\begin{split}
\label{symform}
\tilde{f}_{p.q} &= \int_0^1 \int_0^1 f(x,y) x^p y^q dx dy \\
& = \frac{1}{2}  \int_0^1 \int_0^1 f(x,y) (x^p y^q + x^q y^p) dx dy \\
& =  \int_0^1 \int_0^y f(x,y) (x^p y^q + x^q y^p) dx dy\\
& = \int_0^1 \int_0^y f(x,y) x^p y^q dx dy
+ \int_0^1 \int_0^y f(x,y) x^q y^p dx dy.
\end{split}
\end{align}
Notice that the second integral is equal to the first 
integral, but with $p$ and $q$ interchanged. So next compute the integral 
\begin{align}
\begin{split}
\label{apeqn}
&\int_0^1 \int_0^y f(x,y) x^p y^q dx dy
=  \int_0^1 \int_0^y  \frac{ 8x^2 y^2 ( 5 x^2 - 8 xy + 5 y^2)}{9(x+y)^5} x^p y^q dx dy \\
& = \frac{8}{9} \Big(  5 I(p+4, q+2;5) - 8 I(p+3,q+3;5) + 5 I (p+2, q+4;5) \Big).
\end{split}
\end{align}
Using Proposition \ref{ipqform}, the first term on the second line of \eqref{apeqn} is 
\begin{align}
\begin{split}
\label{fr1}
& \frac{8}{9} \Big( 5 I(p+4, q+2;5) \Big) \\
&=
\frac{8}{9(p+q+3)} \Bigg\{  5 \Bigg( 
\frac{-1}{192} \Big( 4 (p+4)^3 - 10 (p+4)^2 + 8(p+4) +3 \Big)\\
& \ \ + \frac{1}{24} (p+4)(p+3)(p+2)(p+1) (-1)^p (  \log(2) - A(p))
\Bigg) \Bigg\}.
\end{split}
\end{align}
If $p \geq 1$,  using Proposition \ref{ipqform}, the second term on the 
second line of \eqref{apeqn} is  
\begin{align}
\begin{split}
\label{fr2}
& \frac{8}{9} \Big(  - 8 I(p+3,q+3;5)  \Big) \\
&=
\frac{8}{9(p+q+3)} \Bigg\{ -8 \Bigg(
\frac{-1}{192} \Big( 4 (p+3)^3 - 10 (p+3)^2 + 8(p+3) +3 \Big)\\
& \ \ + \frac{1}{24}(p+3)(p+2)(p+1)p (-1)^p (  \log(2) - A(p-1))
\Bigg) \Bigg\}.
\end{split}
\end{align}
If $p \geq 2$,  using Proposition \ref{ipqform}, the third term on the 
second line of \eqref{apeqn} is 
\begin{align}
\begin{split}
\label{fr3}
& \frac{8}{9} \Big( 5 I (p+2, q+4;5) \Big) \\
&=
\frac{8}{9(p+q+3)} \Bigg\{ 5 \Bigg(
\frac{-1}{192} \Big( 4 (p+2)^3 - 10 (p+2)^2 + 8(p+2) +3 \Big)\\
& \ \ + \frac{1}{24}(p+2)(p+1)(p)(p-1) (-1)^p (  \log(2) - A(p-2))
\Bigg)
\Bigg\}.
\end{split}
\end{align}
Using the formulas 
\begin{align}
A(p-1) &= A(p) + (-1)^p \frac{1}{p},\\
A(p-2) & = A(p) + (-1)^{p-1} \frac{1}{p-1} + (-1)^p \frac{1}{p},
\end{align}
and adding together \eqref{fr1}, \eqref{fr2}, and \eqref{fr3}, 
after some algebraic simplification, it follows that 
\begin{align}
\begin{split}
\label{stfo}
& \frac{8}{9} \Big(  5 I(p+4, q+2;5) - 8 I(p+3,q+3;5) + 5 I (p+2, q+4;5) \Big) \\ 
&= \frac{1}{18} \frac{1}{p + q +3} \Big\{ 
-\frac{109}{2} - 68 p - 33 p^2  - 6 p^3  \\
&+ 4(-1)^{p}  (p+1) (p+2) (3p^2 + 9p +10)\big( \log(2) - A(p)\big) \Big\},
\end{split}
\end{align}
which holds assuming that $p \geq 2$, $q \geq 0$.

Next consider the case $p=0$. Using Proposition \ref{ipqform}
it is easy to verify that the left hand side of \eqref{stfo} is equal to 
\begin{align*}
\frac{ - 109}{36(q+3)} + \frac{80 \log(2)}{18(q+3)},
\end{align*}
which is easily seen to be the same as the right hand side of \eqref{stfo}. 

For the case $p=1$,  using Proposition \ref{ipqform}
the left hand side of \eqref{stfo} is easily verified to be equal to 
\begin{align*}
\frac{ 733}{36(q+4)} - \frac{88 \log(2)}{3(q+4)},
\end{align*}
which is easily seen to be the same as the right hand side of \eqref{stfo}.

Consequently, 
\eqref{stfo} holds for all $p \geq 0$ and $q \geq 0$. 
From \eqref{symform}, the desired formula \eqref{tfpq} is
obtained by symmetrizing this expression in $p$ and $q$.

\bibliography{Mass_references}

\end{document}